\documentclass{article}
\usepackage{amsmath,amssymb,mathrsfs,amsthm, bbm}

\theoremstyle{plain}
\newtheorem{thm}{Theorem}
\newtheorem{coro}{Corollary}
\newtheorem{lemma}{Lemma}[section]
\newtheorem{prop}{Proposition}[section]

\theoremstyle{definition}

\newtheorem{asm}{Assumption}

\theoremstyle{remark}
\newtheorem{remark}{Remark}[section]

\newcommand{\RR}{\mathbb{R}}
\newcommand{\CC}{\mathbb{C}}
\newcommand{\TT}{\mathbb{T}}
\newcommand{\ZZ}{\mathbb{Z}}
\newcommand{\Cinf}{C^{\infty}}
\newcommand{\ve}{\varepsilon}
\newcommand{\ellhd}{\ell^2(h\mathbb{Z}^d)}
\newcommand{\supp}{\mathop{\mathrm{supp}}}
\newcommand{\bcdot}{\boldsymbol{\cdot}}
\begin{document}

\title{Semiclassical analysis and the Agmon-Finsler metric for discrete Schr\"{o}dinger operators}
\author{Kentaro Kameoka}

\date{}


\maketitle

\begin{abstract}
The Agmon estimate for multi-dimensional discrete Schr\"{o}dinger operators 
is studied 
with emphasis on the microlocal analysis on the torus. 
We first consider the semiclassical setting 
where semiclassical continuous Schr\"{o}dinger operators are discretized with the mesh width 
proportional to the semiclassical parameter. 
Under this setting, 
the Agmon estimate for eigenfunctions is described by 
an Agmon metric, which is
a Finsler metric rather than a Riemannian metric. 
Klein-Rosenberger\,(2008) 
proved this by a different argument in the case of a potential minimum. 
We also prove the Agmon estimate and the optimal anisotropic 
exponential decay of eigenfunctions 
for discrete Schr\"{o}dinger operators in the non-semiclassical standard setting.
\end{abstract}

\section{Introduction}
We first discuss the semiclassical Agmon estimate for discrete Schr\"{o}dinger operators. 
We start with a continuous semiclassical Schr\"{o}dinger operator
\[
 H^{\mathrm{cont}}=H^{\mathrm{cont}}(h)=-h^2\Delta+V(x) 
\hspace{0.3cm} \text{on} \hspace{0.3cm} L^2(\mathbb{R}^d),
\]
where $V\in C^{\infty}(\mathbb{R}^d; \mathbb{R})$. 
The dimension $d\in \mathbb{Z}_{>0}$ is fixed throughout this paper. 
If we discretize this operator with mesh width $\tau>0$, 
we obtain a discrete Schr\"{o}dinger operator $H^{\tau}(h)$ on $\ell^2(\tau \mathbb{Z}^d)$ 
defined by 
\[
H^{\tau}(h)u(x)=-\left(\frac{h}{\tau}\right)^2\sum_{|x-y|=\tau}(u(y)-u(x))+V(x)u(x), 
\]
where $x, y\in \tau \mathbb{Z}^d \subset \mathbb{R}^d$ and $u\in \ell^2(\tau \mathbb{Z}^d)$. 
A rich quantum-classical correspondence is obtained if 
we discretize $H^{\mathrm{cont}}(h)$ with the mesh width proportional to 
the semiclassical parameter ($\tau \sim h$). 
For simplicity, we put $\tau=h$ and obtain a semiclassical discrete Schr\"{o}dinger operator $H(h)$
on $\ell^2(h\mathbb{Z}^d)$ defined by 
\[
H(h)u(x)=-\sum_{|x-y|=h}(u(y)-u(x))+V(x)u(x), 
\]
where $x, y\in h\mathbb{Z}^d \subset \mathbb{R}^d$ and $u\in \ell^2(h\mathbb{Z}^d)$. 
This setting was studied in \cite{HS} for $d=1$ in the context of the Harper operator
and in \cite{KR1},\,\cite{Ra} for general $d$. 

The limit $\tau \to 0$ for fixed $h>0$ is the problem of the continuum limit 
and various quantities related to $H^{\tau}(h)$ converge to those of $H^{\mathrm{cont}}(h)$, 
that is ``$\lim_{\tau \to 0}H^{\tau}(h)=H^{\mathrm{cont}}(h)$'' 
(see for instance, \cite{IJ}~\cite{NT}). 
In the limit $h \to 0$ for fixed $\tau>0$, $H^{\tau}(h)$ converges to $V(x)$ 
on $\ell^2(\tau \mathbb{Z}^d)$ since difference operators are bounded. 
The related rescaled problem of $h^{-2}H^{\tau}(h)$ when $h\to 0$ for fixed $\tau>0$ 
is studied in~\cite{BPL}. 
It may be interesting to study $\tau=h^{\alpha}$ for $1<\alpha<\infty$. 
The continuum limit formally corresponds to $\alpha=\infty$.

In this paper, the semiclassical discrete Fourier transform 
$\mathcal{F}_{h}: \ell^2(h\mathbb{Z}^d)\to L^2(\mathbb{T}^d)$, 
where $\mathbb{T}^d=\mathbb{R}^d / 2\pi \mathbb{Z}^d$, is defined by 
\[
\mathcal{F}_{h}u(\xi)=(2\pi)^{-d/2}\sum_{x \in h\mathbb{Z}^d}u(x)e^{i\langle x, \xi \rangle/h}.
\]
Then we have 
\[
\widetilde{H}(h)\stackrel{\mathrm{def}}{=}\mathcal{F}_{h}H(h)\mathcal{F}_{h}^{-1}=
\sum_{j=1}^d (2-2\cos\xi_j)+V(hD_{\xi}).
\]
Here $V(hD_{\xi})$ denotes the semiclassical pseudodifferential operator on $\mathbb{T}^d$ with 
the symbol $V(x)$ (see Section 2 for the definition), 
where $x \in \mathbb{R}^d$ is interpreted as the dual variable of 
$\xi \in \mathbb{T}^d$ on $T^*\mathbb{T}^d$.
Thus $\widetilde{H}(h)$ is the semiclassical quantization of 
the classical Hamiltonian 
$p(\xi, x)=\sum_{j=1}^d (2-2\cos\xi_j)+V(x)\in C^{\infty}(T^*\mathbb{T}^d)$ on the torus, 
and it is expected that various quantities related to $H(h)$ are asymptotically described 
in terms of $p(\xi, x)$, that is ``$\lim_{h \to 0}H(h)=p(\xi, x)$''. 
For instance, the Weyl law for eigenvalues naturally follows (Proposition~\ref{thm-1}).

We set $\mathcal{G}_{E}=\{x \in \RR^d|\, V(x)\le E\}$ and 
$\mathcal{G}_{E, \delta}=\{x \in \RR^d|\, \mathrm{dist}(x, \mathcal{G}_{E})<\delta\}$ 
($\mathrm{dist}(\bcdot, \bcdot)$ is the usual Euclidean distance). 
We set $\mathcal{G}_{E, \delta}^c=\RR^d \setminus \mathcal{G}_{E, \delta}$. 
Denote the space of smooth functions which are bounded with their all derivatives by $\Cinf_b(\RR^d)$. 

\begin{asm}\label{asm-Agmon}
The potential $V$ belongs to $\Cinf_b(\RR^d; \RR)$ and there exists $E\in \RR$ such that 
$\inf_{x\in \mathcal{G}_{E, \delta}^c}V(x)>E$ for any $\delta>0$. 
\end{asm}
We note that the compactness of $\mathcal{G}_{E}$ is not assumed. 
We set (\cite{KR1})
\[
L(x, v)=\sup_{\xi\in K_x}\langle \xi, v \rangle,  
\]
where 
\[
K_x=\bigl\{\xi \in \RR^d \bigm|\, \sum_{j=1}^d \sinh^2\frac{\xi_j}{2} \le 
\frac{(V(x)-E)_+}{4}\bigr\}.
\]
Here $(\boldsymbol{\cdot})_+=\max\{\boldsymbol{\cdot}, 0\}$. 
We call $L(x, v)$ or $L: T\RR^d \to [0, \infty)$ 
the Agmon-Finsler metric for discrete Schr\"{o}dinger operators. 
This gives the length of $v \in T_x \RR^d=\RR^d$ in this metric. 
Let $d_E(x, y)$ be the (pseudo-)distance between $x, y\in \RR^d$ induced from $L(x, v)$ 
(see Section~3.2 for details). 
Set 
\[
d_{E}(x)=\inf_{y\in \mathcal{G}_E}d_E(x, y). 
\] 

We state the semiclassical Agmon estimate for discrete Schr\"{o}dinger operators.
\begin{thm}\label{thm-2}
Under Assumption~\ref{asm-Agmon} and the above notation, 
for any $C_0>0$, $\delta_0>0$ and $\ve>0$, 
there exist $C>0$, $h_0>0$, $0<\delta<\delta_0$, $\chi,\, \tilde{\chi}\in \Cinf_b(\RR^d; [0, 1])$ 
with 
\[
\supp (1-\chi)\subset \mathcal{G}_{E, \delta}, \hspace{0.2cm}
\supp \tilde{\chi}\subset \mathcal{G}_{E, \delta}\setminus \mathcal{G}_{E, \delta/2}
\]
and $\rho \in \Cinf(\RR^d; \RR_{\ge 0})$ with 
\[
|\rho(x)-(1-\ve)d_E(x)|\le\ve \hspace{0.2cm}\text{for} 
\hspace{0.2cm} x \in \RR^d
\]
such that for $0<h<h_0$,  
\[
\|\chi e^{\rho(x)/h}u\|_{\ell^2}\le C\|\tilde{\chi} u\|_{\ell^2}
+C\|\chi e^{\rho(x)/h}(H(h)-z)u\|_{\ell^2}
\]
for any $u\in \ellhd$ and any $z\in [E-C_0, E+C_0 h]+i[-C_0, C_0]$. 
\end{thm}
We note that Theorem~\ref{thm-2} also 
provides the exponential decay at infinity of eigenfunctions when $\mathcal{G}_E$ is bounded 
though it is valid only for small $h>0$. The case of $h=1$ is discussed below. 

For $d=1$, the Agmon estimate for $H(h)$ was proved in \cite{HS} using the 
Agmon-type Riemannian metric
\[
ds_E=2\mathrm{arsinh}\frac{\sqrt{(V(x)-E)_+}}{2}ds, 
\]
where $ds$ is the length of the standard metric on $\RR$. 
In higher dimensions, Klein-Rosenberger~\cite{KR1}\,\cite{KR3} introduced the same Finsler metric 
and proved the Agmon estimate in the case of potential minimums, 
where $\mathcal{G}_E$ consists of finite points. 
We allow general $\mathcal{G}_E$, which is possibly unbounded. 
The strategy of the proof in \cite{KR1} 
is similar to that in Dimassi-Sj\"{o}strand~\cite[Section~6]{DS} while 
our proof is similar to that in Nakamura~\cite{N} and is more microlocal. 
Rabinovich~\cite{Ra} also studied the same semiclassical setting for general $d$ 
and proved the Agmon estimate though the relation with Finsler metric is not discussed in \cite{Ra}. 

We note that Klein-Rosenberger~\cite{KR2} 
constructed WKB solutions for the eigenfunction problem of $H(h)$ near 
a nondegenerate potential minimum in terms of the Agmon-Finsler metric, 
which shows that this metric is the natural notion for estimating the tunneling effect for $H(h)$.

We next prove the Agmon estimate and the optimal anisotropic exponential decay of eigenfunctions 
for the non-semiclassical discrete Schr\"{o}dinger operator 
\[
Hu(x)=-\sum_{|x-y|=1}(u(y)-u(x))+V(x)u(x), 
\]
where $x, y\in \mathbb{Z}^d$. Namely, we set $H=H(1)$. 
\begin{asm}\label{asm-Agmon2}
The potential $V: \ZZ^d \to \RR$ has a smooth extension $\widetilde{V}: \RR^d\to \RR$ with the following 
properties.  
There exists $0<\theta\le 1$ such that 
\begin{equation}\label{theta}
|\partial^{\alpha}\widetilde{V}(x)|\le C_{\alpha}(1+|x|)^{-\theta |\alpha|}
\end{equation} 
for any $\alpha \in \ZZ^d_{\ge 0}$ and $\underline{\lim}_{|x|\to \infty}\widetilde{V}(x)\ge 0$. 
\end{asm}

Note that any $V\in \ell^{\infty}_{\mathrm{comp}}(\ZZ^d)$ satisfies the Assumption~\ref{asm-Agmon2}. 
We fix $E<0$. 
We write $\widetilde{V}=V$ without confusion. 
We note that a necessary and sufficient condition for the existence of an 
extension $V:\RR^d \to \RR$ of $V: \ZZ^d\to \RR$ satisfying (\ref{theta}) is given by 
Nakamura~\cite[Lemma~2.1]{N2}. 
Although the case of $\theta=1$ is discussed in \cite{N2}, the case of $0<\theta<1$ is similar.

We set $q(\xi)=4\sum_{j=1}^d \sinh^2\frac{\xi_j}{2}$. 
We also define the Gauss map $G_E:\partial K^E \to \mathbb{S}^{d-1}$ 
of $K^E=\{\xi\in \RR^d |\, q(\xi)\le|E|\}$ by $G_E(x)=\partial q(\xi)/|\partial q(\xi)|$ 
for $\xi\in \partial K^E$. 
This is bijective since $K^E$ is convex and the Gaussian curvature of $\partial K^E$ does not vanish. 
We set 
\[
\rho_E(x)=\sup_{\xi \in K^E}\langle x, \xi \rangle=x\cdot G_E^{-1}\left(\frac{x}{|x|}\right). 
\] 

\begin{thm}\label{thm-3}
Under Assumption~\ref{asm-Agmon2} and the above notation, 
for any $C_0>0$ and $\ve>0$ there exist $C>0$ and 
$1-\chi,\, \tilde{\chi}\in \ell_{\mathrm{comp}}^{\infty}(\ZZ^d)$ 
such that
\[
\|\chi e^{(1-\ve)\rho_E(x)}u\|_{\ell^2}\le C\|\tilde{\chi} u\|_{\ell^2}
+C\|\chi e^{(1-\ve)\rho_E(x)}(H-z)u\|_{\ell^2}
\]
for any $u\in \ell^2(\ZZ^d)$ and any $z\in [E-C_0, E]+i[-C_0, C_0]$. 
\end{thm}

\begin{coro}\label{coro1}
Under Assumption~\ref{asm-Agmon2} and the above notation, 
if $(H-E)u=0$ and $u \in \ell^2(\ZZ^d)$, then 
for any $\ve>0$ there exists $C_{\ve}>0$ such that 
\[
|u(x)|\le C_{\ve}e^{-(1-\ve)\rho_E(x)}
\]
for any $x\in \ZZ^d$.
\end{coro}

\begin{remark}\label{rem-Finsler1}
We note that $\rho_E(x)$ is the length of the line segment joining $0$ and $x$ 
with respect to the Agmon-Finsler metric $L(x, v)$ at energy $E$ for $V \equiv 0$. 
The geodesics with respect to this metric in this case are the straight lines 
since $L(x, v)$ is independent of $x$, and thus $\rho_E(x)$ coincides with $d_E(x, 0)$ for $V\equiv 0$
(see~\cite[Section 5.3, 6.6]{BCZ}).  
\end{remark}

Rabinovich-Roch~\cite{RR} proved the exponential decay of eigenfunctions for the discrete Schr\"{o}dinger 
operator with a slowly oscillating potential. 
In our notation, their exponential decay corresponds to $|u(x)|\le C_{\ve}e^{-(1-\ve)\rho(x)}$ with a condition on 
$\sup_j|\partial_{x_j}\rho(x)|$. 
Our condition $\partial \rho(x)\in K^E$ is more precise and is optimal 
as seen in Subsection~4.2.

The Agmon estimate was introduced by Agmon 
(see~\cite{A}). 
Our approach to Theorem~\ref{thm-2} and Theorem~\ref{thm-3}  
is similar to the arguments in~\cite{N}. 
Since we work in the Fourier space, we need to study the operator conjugated with the 
exponential of a Fourier multiplier and the calculations are more complicated than 
those in~\cite{N}. 
See~\cite{N} for the history of the semiclassical Agmon estimate for continuous 
Schr\"{o}dinger operators.

We note that the viewpoint of microlocal analysis on the torus for discrete Schr\"{o}dinger operators 
is also discussed in the context of long-range scattering theory (see~\cite{N2}~\cite{T}).

In Section 2, we recall basic facts about the microlocal analysis on the torus and present 
the Weyl law.  
In Section 3, we discuss the Agmon-Finsler metric for discrete 
Schr\"{o}dinger operators and prove Theorem~\ref{thm-2}. 
In Section~3, we prove the Agmon estimate for non-semiclassical discrete Schr\"{o}dinger operators 
and the optimality of this estimate.

\section{Preliminaries}
In this section, we recall basic facts on microlocal analysis on the torus.
We identify functions on $T^*\TT^d$ or $\TT^d$ with those on $T^*\RR^d$ or $\RR^d$ which are 
$2\pi\ZZ^d$-periodic. 
We recall the notation $\langle x \rangle=(1+x^2)^{1/2}$ and 
\[
S^m_{\theta, 0}(T^*\TT^d)=\{a(\bcdot; h)\in \Cinf(T^* \TT^d)|\, 
|\partial^{\alpha}_{\xi}\partial^{\beta}_x a(\xi, x; h)|
\le C_{\alpha, \beta}\langle x \rangle^{m-\theta|\beta|}\}.
\]
Here $\alpha$ and $\beta$ range over $\ZZ^d_{\ge 0}$. 
We write $S^m_{\theta, 0}=S^m_{\theta, 0}(T^*\TT^d)$, $S^m=S^m_{1, 0}$, 
 $S=S^0_{0, 0}$ and $S^{-\infty}=\bigcap_{m\in \RR}S^m$.
For $a\in S^m_{\theta, 0}$ and for $u\in C^{\infty}(\TT^d)$, 
we define
\[
a(\xi, hD_{\xi})u(\xi)=(2\pi h)^{-d}\int_{\RR^d}\int_{\RR^d}a(\xi, x)
e^{i\langle \xi-\eta, x\rangle/h}u(\eta)d\eta dx
\]
in the sense of oscillatory integral. 
The corresponding class of pseudodifferential operators is denoted by $\mathrm{Op}S^m_{\theta, 0}$. 
\begin{lemma}
$V(hD_{\xi})=\mathcal{F}_h V(x) \mathcal{F}_h^{-1}$ for $V\in \Cinf_b(\RR^d)$.
\end{lemma}
\begin{proof}
We have 
\begin{align*}
V(hD_{\xi})u(\xi)
&=(2\pi h)^{-d}\int_{\RR^d}\int_{\RR^d}V(x)e^{i\langle \xi-\eta, x\rangle/h}u(\eta)d\eta dx \\
&=(2\pi h)^{-d}\int_{\RR^d}V(x)\left(\sum_{x\in h\ZZ^d}(2\pi)^{d/2}h^d (\mathcal{F}_h^{-1}u)(x)
\delta_x\right)e^{i\langle \xi, x \rangle/h}dx\\
&=(2\pi)^{-d/2}\sum_{x\in h\ZZ^d}V(x)(\mathcal{F}_h^{-1}u)(x)e^{\langle\xi, x \rangle/h}
\end{align*}
for $u \in \Cinf(\TT^d)$, which completes the proof.
\end{proof}

Although we use the spacial structure of the torus to define $a(\xi, hD_{\xi})$, 
we can employ the general theory of pseudodifferential operators on manifolds 
including the functional calculus and the trace formula for pseudodifferential operators 
(see~\cite[Chapter 5, 14]{Z}). 
To illustrate these, we give the Weyl law. 
\begin{prop}\label{thm-1}
Assume that $V\in \Cinf(\RR^d; \RR)$, $\underline{\lim}_{|x|\to \infty}V(x)\ge 0$ 
and there exists $0<\theta\le 1$ such that 
\begin{align*}
|\partial^{\alpha}V(x)|\le C_{\alpha}(1+|x|)^{-\theta |\alpha|}
\end{align*} 
for any $\alpha \in \ZZ^d_{\ge 0}$. 
Then for any fixed $a<b<0$, the number $N_{[a, b]}(h)$ of eigenvalues of $H(h)$ in $[a, b]$ satisfies 
\[
N_{[a, b]}(h)=(2\pi h)^{-d}\mathrm{Vol}(\{ (\xi, x)\in T^*\mathbb{T}^d |\, a \le p(\xi, x) \le b\})
+o(h^{-d})
\]
when $h\to 0$.
\end{prop}
\begin{proof}[Proof of Proposition~\ref{thm-1}]
Take small $\varepsilon>0$ and $\chi_{1, \varepsilon}, \chi_{2, \varepsilon}\in C_c^{\infty}(\RR; [0, 1])$ 
such that $\chi_{1, \ve}=1$ on $[a-\ve, b+\ve]$, $\supp\chi_{1, \ve}\subset [a-2\ve, b+2\ve]$, 
$\chi_{2, \ve}=1$ on $[a+2\ve, b-2\ve]$ and $\supp\chi_{2, \ve}\subset [a+\ve, b-\ve]$. 
Then we have
\[
\mathrm{tr} (\chi_{2, \varepsilon}(\widetilde{H}(h)))\le N_{[a, b]}(\widetilde{H}(h))
\le \mathrm{tr} (\chi_{1, \varepsilon}(\widetilde{H}(h)))
\]
since $N_{[a, b]}(\widetilde{H}(h))=\mathrm{tr} (\chi_{[a, b]}(\widetilde{H}(h)))$ 
and $\chi_{2, \varepsilon}\le \chi_{[a, b]} \le \chi_{1, \varepsilon}$.

The functional calculus and the trace formula for pseudodifferential operators imply that 
\[
\mathrm{tr} (\chi_{j, \varepsilon}(\widetilde{H}(h)))
=(2\pi h)^{-d}\int_{T^*\TT^d}\chi_{j, \varepsilon}(p(\xi, x))d\xi dx+\mathcal{O}_{\ve}(h^{-d+1})
\]
for $j=1, 2$. 
We note that $\mathrm{Vol}_{2d}(\{(\xi, x)|\, p(\xi, x)=a, b\})=0$, which follows from Fubini's theorem 
and the definition of $p(\xi, x)$. 
Then we have 
\[
\lim_{\ve \to 0}\int_{T^*\TT^d}\chi_{j, \varepsilon}(p(\xi, x))d\xi dx
=\mathrm{Vol}(\{ (\xi, x)\in T^*\mathbb{T}^d |\, a \le p(\xi, x) \le b\})
\]
for $j=1, 2$. 

Take any $\delta>0$. Then for sufficiently small $\ve>0$, the above arguments imply that 
\[
-\delta-\mathcal{O}_{\ve}(h)\le (2\pi h)^d N_{[a, b]}(\widetilde{H}(h))
-\mathrm{Vol}(\{ (\xi, x)|\, a \le p(\xi, x) \le b\}) 
\le \delta+\mathcal{O}_{\ve}(h).
\]
Taking $h\to 0$ and then taking $\delta \to 0$, the proof is finished.
\end{proof}
The proof followed the standard strategy (see~\cite{DS}).

\section{Semiclassical Agmon estimate}
In this section, we prove Theorem~\ref{thm-2}.
\subsection{Calculation of exponentially conjugated operator}
Take any $\rho \in \Cinf_b(\RR^d; \RR)$. 
We compute $\widetilde{H}_{\rho}(h)=e^{\rho(hD_{\xi})/h}\widetilde{H}(h)e^{-\rho(hD_{\xi})/h}$. 
Since we have $e^{\rho(hD_{\xi})/h}V(hD_{\xi})e^{-\rho(hD_{\xi})/h}=V(hD_{\xi})$, 
we only have to consider $e^{\rho(hD_{\xi})/h}p_0(\xi)e^{-\rho(hD_{\xi})/h}$, 
where $p_0(\xi)=\sum_{j=1}^d (2-2\cos\xi_j)$. 

\begin{lemma}\label{lem-symbol}
For $\rho \in \Cinf_b(\RR^d; \RR)$, 
\[
e^{\rho(hD_{\xi})/h}p_0(\xi)e^{-\rho(hD_{\xi})/h}=a_{\rho}(\xi, hD_{\xi}; h)\in \mathrm{Op}S, 
\]
where $a_{\rho}\sim \sum_{k=0}^{\infty}h^k a_{\rho, k}(\xi, x)$ with $a_{\rho, k}\in S$ and 
\[
a_{\rho, 0}(\xi, x)=p_0(\xi-i\partial \rho(x), x).
\]
If moreover 
\begin{equation}\label{asm-rho}
|\partial^{\alpha}_x \rho(x)|\le C_{\alpha}\langle x \rangle^{1-|\alpha|}
\hspace{0.2cm}\text{for any}\hspace{0.2cm}\alpha \in \ZZ^d_{\ge 0},
\end{equation}
then $a_{\rho}\in S^0$ and $a_{\rho, k}\in S^{-k}$. 
\end{lemma}

\begin{proof}
We set $g(x)=e^{-|x|^2}$. 
Then we have 
\begin{align*}
e^{-\rho(hD_{\hat{\eta}})/h}u(\hat{\eta})
&=\lim_{\ve \to 0}(2\pi h)^{-d}\int_{\RR^{2d}}e^{i\langle \hat{\eta}-\eta, x\rangle/h}
  e^{-\rho(x)/h}u(\eta)g(\ve x)g(\ve \eta)dx d\eta\\
&=(2\pi h)^{-d}\int_{\RR^{2d}}e^{i\langle \hat{\eta}-\eta, x\rangle/h}\hspace{0.1cm} {}^{t}L_1^N 
  \bigl(e^{-\rho(x)/h}u(\eta)\bigr)dx d\eta, 
\end{align*}
where $N \ge 2d+1$ and
\[
L_1=\frac{1-x\cdot hD_{\eta}+(\hat{\eta}-\eta)\cdot hD_x}{1+|x|^2+|\eta-\hat{\eta}|^2}.
\]
Thus 
\begin{align*}
&e^{\rho(hD_{\xi})/h}p_0(\xi)e^{-\rho(hD_{\xi})/h}u(\xi)\\
&=(2\pi h)^{-2d}\int_{\RR^{2d}}e^{i\langle \xi-\hat{\eta}, y\rangle/h}
  \hspace{0.1cm} {}^{t}L_2^N e^{\rho(y)/h}p_0(\hat{\eta})
\int_{\RR^{2d}}e^{i\langle \hat{\eta}-\eta, x\rangle/h}\\ 
              &\hspace{6.5cm} {}^{t}L_1^{2N} 
  \bigl(e^{-\rho(x)/h}u(\eta)\bigr)dx d\eta dy d\hat{\eta}\\
&=\lim_{\ve \to 0} (2\pi h)^{-2d}\int_{\RR^{4d}}e^{i\langle \xi-\hat{\eta}, y\rangle/h}
  e^{\rho(y)/h}e^{i\langle \hat{\eta}-\eta, x\rangle/h}
  e^{-\rho(x)/h}u(\eta)\\
              &\hspace{6.2cm}g(\ve x)g(\ve \eta)g(\ve y)g(\ve \hat{\eta})
  dx d\eta dy d\hat{\eta}, 
\end{align*}
where 
\[
L_2=\frac{1-y\cdot hD_{\hat{\eta}}+(\xi-\hat{\eta})\cdot hD_y}{1+|y|^2+|\hat{\eta}-\xi|^2}.
\]
We set $\rho(y)-\rho(x)=(y-x)\cdot \Phi(x, y)$ with $\Phi(x, y)=\int_0^1 \partial \rho(y+t(x-y))dt$. 
We deform the integral and obtain 
\begin{align*}
&e^{\rho(hD_{\xi})/h}p_0(\xi)e^{-\rho(hD_{\xi})/h}u(\xi)\\
&=\lim_{\ve \to 0}(2\pi h)^{-2d}\int_{\RR^{4d}}
  e^{i\langle \xi-\hat{\eta}, y\rangle/h}e^{i\langle \hat{\eta}-\eta, x\rangle/h}
  p_0(\hat{\eta}-i\Phi(x,y))u(\eta)\\
   &\hspace{5.0cm}   g(\ve x)g(\ve \eta)g(\ve y)g(\ve \hat{\eta}-\ve i\Phi(x, y))
   d\eta dx d\hat{\eta} dy.
\end{align*}
Using ${}^{t}L_2^N$ and ${}^{t}L_1^{2N}$, we see that 
\begin{align*}
&e^{\rho(hD_{\xi})/h}p_0(\xi)e^{-\rho(hD_{\xi})/h}u(\xi)\\
&=\lim_{\ve \to 0}\lim_{\ve' \to 0}
  (2\pi h)^{-2d}\int_{\RR^{4d}}e^{i\langle \xi-\hat{\eta}, y\rangle/h}e^{i\langle \hat{\eta}-\eta, x\rangle/h}
  p_0(\hat{\eta}-i\Phi(x,y))u(\eta)\\
   &\hspace{5.0cm}   \psi(\ve x)\psi(\ve \eta)\psi(\ve' y)\psi(\ve' \hat{\eta})
   d\eta dx d\hat{\eta} dy\\
&=\lim_{\ve \to 0}\lim_{\ve' \to 0}
  (2\pi h)^{-2d}\int_{\RR^{4d}}e^{i \langle \xi-\eta, x \rangle/h}
  e^{-i\langle y, \hat{\eta}\rangle/h}p_0(\hat{\eta}+\xi-i\Phi(x,y+x))u(\eta)\\
   &\hspace{3.0cm}  \psi(\ve x)\psi(\ve \eta)\psi(\ve' y+\ve'x)
   \psi(\ve'\hat{\eta}+\ve'\xi)d\hat{\eta} dy d\eta dx, 
\end{align*}
where $\psi \in \Cinf_c(\RR^d)$ is a cutoff near $0$ with 
$\supp \psi \subset \{x\in \RR^d|\,|x|<1/4\}$. 
We also changed the variables from $y$ and $\hat{\eta}$ to $y+x$ and $\hat{\eta}+\xi$, respectively. 

We next insert 
\[
1=(1-\psi(y)\psi(\hat{\eta}))+\psi(y)\psi(\hat{\eta})
\] 
into the integrand and estimate the
$\lim_{\ve' \to 0}(2\pi h)^{-d}\int_{\RR^{2d}}\cdots d\hat{\eta} dy$ part. 
We set 
\[
L_3=\frac{-\hat{\eta}hD_y-yhD_{\hat{\eta}}}{|\hat{\eta}|^2+|y|^2}. 
\]
We see that the $1-\psi(y)\psi(\hat{\eta})$ term contributes as $h^{\infty}S$ 
if we use ${}^{t}L_3^N$ with $N \gg 1$. 
To estimate the $\psi(y)\psi(\hat{\eta})$ term, 
we apply the stationary phase method (\cite[Theorem~7.7.6]{H1}) with respect to $(\hat{\eta}, y)$. 
The stationary point ($\partial_{\hat{\eta}, y}\phi=0$) 
is $(\hat{\eta}, y)=(0, 0)$ and we have $\mathrm{sgn}\partial_{\hat{\eta}, y}^2\phi=0$ 
and $|\det \partial_{\hat{\eta}, y}^2\phi|=1$ there. 
We then obtain an asymptotic expansion with respect to $h$ in $S$ 
with the leading term $p_0(\xi-i\Phi(x, x))$. Note that $\Phi(x, x)=\partial \rho(x)$. 

Finally we assume (\ref{asm-rho}) and prove the asymptotic expansion in $S^0$. 
For this, we change the variables from $y$ to $\langle x \rangle y$ and introduce 
$\tilde{h}=h\langle x\rangle^{-1}$. 
Then we have 
\begin{align*}
&e^{\rho(hD_{\xi})/h}p_0(\xi)e^{-\rho(hD_{\xi})/h}u(\xi)\\
&=\lim_{\ve \to 0}\lim_{\ve' \to 0}(2\pi h)^{-d}(2\pi \tilde{h})^{-d}\int_{\RR^{4d}}
 e^{i \langle \xi-\eta, x \rangle/h} e^{-i\langle y, \hat{\eta}\rangle/\tilde{h}}
 p_0(\hat{\eta}+\xi-i\Phi(x,\langle x \rangle y+x))\\
&\hspace{2.5cm}  u(\eta)\psi(\ve x)\psi(\ve \eta)\psi(\ve' \langle x \rangle y+\ve'x)
\psi(\ve'\hat{\eta}+\ve'\xi) d\hat{\eta} dy d\eta dx. 
\end{align*}
We insert 
\[
1=(1-\psi(y))+\psi(y)(1-\psi(\hat{\eta}))+\psi(y)\psi(\hat{\eta})
\] 
into the integrand and estimate the
$\lim_{\ve' \to 0}(2\pi \tilde{h})^{-d}\int_{\RR^{2d}}\cdots d\hat{\eta} dy$ part. 
We set 
\[
\tilde{L}_3=\frac{-\hat{\eta}\tilde{h}D_y-y\tilde{h}D_{\hat{\eta}}}{|\hat{\eta}|^2+|y|^2} 
\hspace{0.2cm}\text{and}\hspace{0.2cm} 
\tilde{L}_4=\frac{-y\tilde{h}D_{\hat{\eta}}}{|y|^2}. 
\]
We see that the $1-\psi(y)$ term contributes as $h^{\infty}S^{-\infty}$ 
if we use ${}^{t}\tilde{L}_3^{d+1}$ and ${}^{t}\tilde{L}_4^{N}$ with $N \gg 1$. 
We also see that the $\psi(y)(1-\psi(\hat{\eta}))$ term contributes as $h^{\infty}S^{-\infty}$ 
if we use ${}^{t}\tilde{L}_3^{N}$ with $N \gg 1$. 
To see this, we note that 
\[
|\partial_y^{\alpha}\Phi(x, \langle x \rangle y+x)|\le C_{\alpha} 
\hspace{0.2cm}\text{for any}\hspace{0.2cm}\alpha \in \ZZ^d_{\ge 0}
\] 
since $|\langle x \rangle y+x|\ge |x|/2$ for $|x|\ge 1$ and $|y|\le 1/4$. 
We apply the stationary phase method to the $\psi(y)\psi(\hat{\eta})$ term 
and obtain asymptotic expansion with respect to $h\langle x \rangle^{-1}$ in $S^0$ by the 
above estimate on $\partial_y^{\alpha}\Phi(x, \langle x \rangle y+x)$. 
These complete the proof.
\end{proof} 

\begin{remark}
The use of the Gaussian weight to justify contour deformation in the oscillatory integral 
is found in the context of the resonance theory (see Galkowski-Zworski~\cite[Appendix~B.1]{GZ}).
\end{remark}

\begin{remark}
The second part of Lemma~\ref{lem-symbol} is used in Section~4.
\end{remark}

This lemma implies that the semiclassical principal symbol is given by 
\[
\sigma_h(\widetilde{H}_{\rho}(h))=p(\xi-i\partial\rho(x), x).
\] 

In the proof of the Agmon estimate, we treat unbounded $\rho\in \Cinf(\RR^d; \RR)$ 
such that $\rho$ is lower semibounded and $\partial \rho \in \Cinf_b(\RR^d, \RR^d)$. 
Take $\nu(t) \in \Cinf(\RR; \RR)$ with $0\le \nu'(t)\le 1$ and $\nu''(t)\le 0$ 
such that $\nu(t)=t$ for $t<0.9$ and $\nu(t)=1$ for $t>1.1$. 
We set $\rho_M(x)=M\nu(\rho(x)/M)$. 
We note that $\rho_M(x)\nearrow \rho(x)$ when $M\to \infty$ since $\nu''(t)\le 0$.

The proof of Lemma~\ref{lem-symbol} implies that 
the first statement in Lemma~\ref{lem-symbol} with $\rho$ replaced by $\rho_M$ is valid uniformly for $M>1$ 
since $\partial \rho_M \in \Cinf_b(\RR^d, \RR^d)$ uniformly for $M>1$. 
The second statement with $\rho$ replaced by $\rho_M$ is also valid uniformly for $M>1$ 
if we add the assumption that $\rho(x)\gtrsim |x|$ for large $|x|$ 
to ensure that (\ref{asm-rho}) with $\rho$ replaced by $\rho_M$ 
is valid uniformly for $M>1$. 

We set 
\[
\widetilde{H}_{M}(h)=e^{\rho_M(hD_{\xi})}\widetilde{H}(h)e^{-\rho_M(hD_{\xi})}.
\]
It may be possible to prove that 
$\widetilde{H}_{\rho}(h)=\lim_{M \to \infty}\widetilde{H}_{M}(h)\in \mathrm{Op}S$ and that 
this is given by the integral expression in the proof of Lemma~\ref{lem-symbol}.  
In fact, in the proof of the Agmon estimate, we do not use this and 
we take the limit $M\to \infty$ in a later step of the proof.

\subsection{The Agmon-Finsler metric}
We recall that 
\[
p_0(\xi)=\sum_{j=1}^d (2-2\cos\xi_j)=4\sum_{j=1}^d \sin^2\frac{\xi_j}{2}.
\] 
We will find a condition which ensures that the real part of 
\[
p_0(\xi-i\partial \rho(x))+V(x)-E
\]
is positive away from $\mathcal{G}_E=\{x \in \RR^d|\, V(x)\le E\}$. 
Since
\begin{align*}
4\sin^2\frac{\xi+i\lambda}{2}
&=4(\sin\frac{\xi}{2}\cosh\frac{\lambda}{2}+i\cos\frac{\xi}{2}\sinh\frac{\lambda}{2})^2,
\end{align*}
we have 
\[
\mathrm{Re}\,(4\sin^2\frac{\xi+i\lambda}{2})\ge -4 \sinh^2\frac{\lambda}{2}.
\]
This implies that 
\begin{equation}\label{realpart}
\mathrm{Re}\,\bigl(p_0(\xi-i\partial \rho(x))+V(x)-E\bigr) \ge 
V(x)-E-4\sum_{j=1}^d \sinh^2\frac{\partial_j \rho(x)}{2}.
\end{equation}

We set 
\[
K_x=\bigl\{\xi \in \RR^d \bigm|\, \sum_{j=1}^d \sinh^2\frac{\xi_j}{2} \le 
\frac{(V(x)-E)_+}{4}\bigr\}, 
\]
which is interpreted as a subset of $T^*_x \RR^d$. 

We present a construction (which is valid for more general $K_x$) of a function $d(x)$ such that 
\[
\partial d(x) \in K_x
\] 
for (almost all) $x\in \RR^d$. 
For this, we define a Finsler metric as the supporting function of $K_x$ (\cite{KR1}); 
\[
L(x, v)=\sup_{\xi\in K_x}\langle \xi, v \rangle,  
\]
which gives the length of $v \in T_x \RR^d=\RR^d$ in this metric. 

\begin{remark}
We note that $K_x$ for $x$ with $V(x)>E$ is a strictly convex compact set with $0\in K_x$ 
such that $\partial K_x$ is smooth and has non-vanishing Gaussian curvature. 
This implies that $\bigl(\frac{1}{2}\partial_{v_i}\partial_{v_j}L(x, v)^2\bigr)_{ij}$ is positive 
definite for $v\not =0$ and $x$ with $V(x)>E$. 
Thus $L(x, v)$ satisfies the conditions of the definition of the 
Finsler metric (\cite[Section~1.1]{BCZ}) on $\mathcal{G}_E^c=\{x \in \RR^d|\, \,V(x)>E\}$. 
\end{remark}

We set 
\[
d_E(x, y)=\inf_{x(\bcdot)}\int_0^1 L(x(t), x'(t))dt, 
\] 
where $x(\bcdot): [0, 1] \to \RR^d$ ranges over $C^1$ curves such that $x(0)=x$ and 
$x(1)=y$. Note that $d_E(x, y)=d_E(y, x)$ since $L(x, v)=L(x, -v)$. 
Take any closed set $\mathcal{G}$ in $\RR^d$. We set 
\[
d_{\mathcal{G}}(x)=d_{E, \mathcal{G}}(x)=\inf_{y\in \mathcal{G}}d_E(x, y). 
\]
Note that $d_{\mathcal{G}}$ is a Lipschitz continuous function. We then have the following. 
\begin{lemma}\label{lem-grad}
For almost all $x \in \RR^d$, 
\[
\partial d_{\mathcal{G}}(x) \in K_x.
\] 
\end{lemma}
\begin{proof}
Take $x$ such that $d_{\mathcal{G}}(x)$ is differentiable at $x$. 
Take any $v \in T_x \RR^d$. By the triangle inequality, we have 
\[
|d_{\mathcal{G}}(x)-d_{\mathcal{G}}(x+tv)|/t\le d_E(x, x+tv)/t.
\]
Taking limit $t\to 0$, we obtain 
\[
|\langle \partial d_{\mathcal{G}}(x), v\rangle|\le L(x, v).
\]
Recall that the compact convex set $K_x$ is recovered from its supporting function as 
\[
K_x=\{\xi \in \RR^d |\,\langle \xi, v \rangle\le L(x, v) \hspace{0.15cm}\text{for any}
\hspace{0.15cm} v\in \RR^d\}
\]
(see~\cite[Section 4.3]{H1}). This implies $\partial d_{\mathcal{G}}(x) \in K_x$.
\end{proof} 

Then the exponential decay of the eigenfunctions is stated in terms of 
\[
d_E(x)=d_{E, \mathcal{G}_E}(x). 
\]
By the inequality~\eqref{realpart} and Lemma~\ref{lem-grad}, we have 
\[
\mathrm{Re}\,\bigl(p_0(\xi-i\partial d_{E}(x))+V(x)-E\bigr) \ge 0.
\]
outside $\mathcal{G}_E$.

\subsection{Proof of Theorem~\ref{thm-2}}
In the proof of Theorem~\ref{thm-2}, we should modify $d_E(x)$ as follows. 
For a given $\ve>0$, we take a sufficiently small $\delta>0$. 
In the following, we fix $\psi_{\delta}\in \Cinf_c(\RR; \RR_{\ge 0})$ such that 
$\supp \psi_{\delta}\subset \{x\in \RR^d|\,  |x|<\delta/30\}$ and 
$\int_{\RR^d} \psi_{\delta}(x)dx=1$. 
Set $\chi=\mathbbm{1}_{\mathcal{G}_{E, \frac{3}{4}\delta}^c}*\psi_{\delta}$, 
$\chi_1=\mathbbm{1}_{\mathcal{G}_{E, \frac{1}{2}\delta}^c}*\psi_{\delta}$ and 
$\tilde{\chi}=\mathbbm{1}_{\mathcal{G}_{E, \frac{7}{8}\delta}
\setminus\mathcal{G}_{E, \frac{5}{8}\delta}}*\psi_{\delta}$. 
Here $\mathbbm{1}$ denotes the indicator function of a set. 
Then $\chi, \chi_1, \tilde{\chi}\in \Cinf_b(\RR^d; [0, 1])$, 
$\chi\chi_1=\chi$, $\tilde{\chi}\partial \chi=\partial \chi$ and 
\[
\supp (1-\chi)\subset \mathcal{G}_{E, \delta}, \hspace{0.2cm}
\supp \tilde{\chi}\subset \mathcal{G}_{E, \delta}\setminus \mathcal{G}_{E, \delta/2}.
\]
By mollifying $(1-\ve)d_{E, \mathcal{G}_{E, \delta}}$, 
we obtain $\rho \in \Cinf(\RR^d; \RR_{\ge 0})$ satisfying   
$|\rho(x)-(1-\ve)d_E(x)|\le\ve$ for $x \in \RR^d$ and 
$\partial\rho(x)\in (1-\ve/2)K_x$ on $\supp\chi_1$.
Moreover, $\mathrm{dist}(\supp\rho, \supp\partial \chi)>\delta/10$. 
 
Define $\rho_M$ and $\widetilde{H}_{M}(h)$ from this $\rho$ as in subsection~3.1. 
Then $\partial\rho_{M}(x)\in (1-\ve/2)K_x$ on $\supp\chi_1$. 

\begin{proof}[Proof of Theorem~\ref{thm-2}]
Take any $z\in [E-C_0, E+C_0 h]+i[-C_0, C_0]$ for a fixed $C_0$. 
Lemma~\ref{lem-symbol} implies that 
\begin{align*}
\chi_1(hD_{\xi})(\widetilde{H}_{M}(h)-z)^*
(\widetilde{H}_{M}(h)-z)\chi_1(hD_{\xi})-\gamma^2\chi_1(hD_{\xi})^2
\end{align*}
belongs to $\mathrm{Op}S$ uniformly for $M>1$ and its principal symbol is 
\[
\chi_1(x)^2|p(\xi-i\partial\rho_M(x), x)-z|^2-\gamma^2\chi_1(x)^2.
\]
Then the inequality~\eqref{realpart} and the estimate for $\partial\rho_M(x)$ above 
imply that this is nonnegative for small $\gamma>0$ 
(we replace $z$ with $z-C_0 h$ if $E\le \mathrm{Re}z \le E+C_0 h$).

Thus the G\r{a}rding inequality implies that there exists $h_0>0$ such that   
\[
\|(\widetilde{H}_{M}(h)-z)\chi_1(hD_{\xi})\hat{u}\|_{L^2(\TT^d)}\ge
\gamma\|\chi_1(hD_{\xi})\hat{u}\|_{L^2(\TT^d)}-\frac{\gamma}{2}\|\hat{u}\|_{L^2(\TT^d)}
\]
for any $\hat{u}\in L^2(\TT^d)$ and $0<h<h_0$. 
Here $h_0$ is independent of $M>1$ by the uniformity mentioned above. 
Replacing $\hat{u}$ with $\chi(hD_{\xi})\hat{u}$, this implies 
\[
\|e^{\rho_M(x)/h}(H(h)-z)e^{-\rho_M(x)/h}\chi u\|_{\ell^2}\ge \frac{\gamma}{2}\|\chi u\|_{\ell^2}
\]
for $u\in \ellhd$ and $0<h<h_0$. Replacing $u$ with $e^{\rho_M(x)/h}u$, we obtain 
\[
\|e^{\rho_M(x)/h}(H(h)-z)\chi u\|_{\ell^2}\ge \frac{\gamma}{2}\|e^{\rho_M(x)/h}\chi u\|_{\ell^2}
\]
for $u\in \ellhd$ and $0<h<h_0$. 
Taking the limit $M\to \infty$, this is valid with $\rho_M(x)$ replaced by $\rho(x)$.
Thus we have
\begin{align*}
\|\chi e^{\rho(x)/h}u\|_{\ell^2}
&\le C \|e^{\rho(x)/h}(H(h)-z)\chi u\|_{\ell^2} \\
&\le C \|\chi e^{\rho(x)/h}(H(h)-z) u\|_{\ell^2}+C\|e^{\rho(x)/h}[H(h), \chi]u\|_{\ell^2} \\
&\le C \|\chi e^{\rho(x)/h}(H(h)-z) u\|_{\ell^2}+C\|\tilde{\chi} u\|_{\ell^2}.
\end{align*}
The last inequality follows from the fact that $\rho=0$ near $\supp\partial\chi$ 
and $\tilde{\chi}=1$ near $\supp \partial \chi$. 
\end{proof}

\section{The exponential decay of eigenfunctions for discrete Schr\"{o}dinger operators}

\subsection{Proof of Theorem~\ref{thm-3}}

Recall that $q(\xi)=4\sum_{j=1}^d \sinh^2\frac{\xi_j}{2}$, 
$K^E=\{\xi\in \RR^d |\, q(\xi)\le|E|\}$, 
$G_E(x)=\partial q(\xi)/|\partial q(\xi)|$ 
for $\xi\in \partial K^E$ and 
\[
\rho_E(x)=\sup_{\xi \in K^E}\langle x, \xi \rangle=x\cdot G_E^{-1}\left(\frac{x}{|x|}\right). 
\] 
\begin{lemma}\label{eikonal}
The function $\rho_E(x)$ satisfies the eikonal equation 
\[
q(\partial \rho_E(x))=|E| \hspace{0.2cm} \text{for any} \hspace{0.2cm} x\in \RR^d\setminus\{0\}.
\]
\end{lemma}
\begin{proof}
By the definition of $G_E$, we have 
\[
q(G_E^{-1}(x/|x|))=|E| \hspace{0.2cm} \text{and} \hspace{0.2cm} (\partial q)(G_E^{-1}(x/|x|))=x/|x|.
\] 
Differentiating the first equality and using the second, we obtain 
\[
x\cdot \partial_{x_j} (G_E^{-1}(x/|x|))=0.
\]
This and the definition of $\rho_E$ imply that 
\[
\partial \rho_E(x)=G_E^{-1}(x/|x|)
\]
and thus 
\[
q(\partial \rho_E(x))=q(G_E^{-1}(x/|x|))=|E|.
\]
\end{proof}

\begin{remark}\label{rem-Finsler2}
Set $\Lambda_0=T^*_0 \RR^d\cap\{q(\xi)=|E|\}$, which is a ($d-1$)-dimensional isotropic submanifold of 
$T^*\RR^d$. Then the solution in Lemma~\ref{eikonal} corresponds to the Lagrangian submanifold 
$\Lambda=\bigcup_{t>0}\Lambda_t$, where $\Lambda_t$ is the image of $\Lambda_0$ 
under the time $t$ map of the Hamilton flow generated by $q(\xi)$. 
\end{remark}

\begin{proof}[Proof of Theorem~\ref{thm-3}]
Take a smooth modification $\tilde{\rho}_E(x)$ of $\rho_E(x)$ such that 
$\tilde{\rho}_E(x)=\rho_E(x)$ for $|x|>1$. 
We see that 
$|\partial^{\alpha}\tilde{\rho}_E(x)|\le C_{\alpha}\langle x \rangle^{-|\alpha|}$ 
for any $\alpha \in \ZZ^d_{\ge 0}$. 
We also note that $\tilde{\rho}_E(x)\gtrsim |x|$ for large $|x|$. 
For a given small $\ve>0$, we define 
$\rho_M$ and $\widetilde{H}_{M}=\widetilde{H}_{M}(1)$ 
from $(1-\ve)\tilde{\rho}_E$ as in subsection~3.1.

Take any $z\in [E-C_0, E]+i[-C_0, C_0]$ for a fixed $C_0$. 
We also take $\chi_1\in \Cinf(\RR^d; [0, 1])$ such that $\supp \chi_1 \subset \{x\in \RR^d|\,|x|>R-2\}$ 
and $\chi_1(x)=1$ for $|x|>R-1$. 
Then Lemma~\ref{lem-symbol} implies that 
\begin{align*}
\chi_1(D_{\xi})(\widetilde{H}_M-z)^*
(\widetilde{H}_M-z)\chi(D_{\xi})-\gamma^2\chi_1(D_{\xi})^2 
\end{align*}
belongs to $\mathrm{Op}S^0_{\theta, 0}$ uniformly for $M>1$ and its principal symbol is 
\[
\chi_1(x)^2|p(\xi-i\partial \rho_M(x), x)-z|^2-\gamma^2\chi_1(x)^2, 
\]
where $0<\theta\le 1$ is that in Assumption~\ref{asm-Agmon2}. 
If $R>2$ is sufficiently large and $\gamma>0$ is sufficiently small, 
this is everywhere nonnegative for any $M>1$
by the construction of $\tilde{\rho}_E$, Lemma~\ref{eikonal} and Assumption~\ref{asm-Agmon2}.

Then the sharp G{\aa}rding inequality implies that 
\begin{align*}
\|(\widetilde{H}_M-z)\chi_1(D_{\xi})\hat{u}\|^2_{L^2}-\gamma^2 \|\chi_1(D_{\xi})\hat{u}\|^2_{L^2}
\ge -C\|\hat{u}\|^2_{H^{-\theta/2}} 
\end{align*}
for any $\hat{u}\in L^2(\TT^d)$. Here $H^{-\theta/2}$ denotes the Sobolev space on $\TT^d$.
We replace $\hat{u}$ with $\chi(D_{\xi})\hat{u}$, where 
$\chi\in \Cinf(\RR^d; [0, 1])$ satisfies $\supp \chi \subset \{x\in \RR^d|\,|x|>R\}$ 
and $\chi(x)=1$ for $|x|>R+1$. 
Then we have 
\[
\|(\widetilde{H}_{M}-z)\chi(D_{\xi})\hat{u}\|^2_{L^2}-\gamma^2 \|\chi(D_{\xi})\hat{u}\|^2_{L^2}
\ge -C\|\chi(D_{\xi})\hat{u}\|^2_{H^{-\theta/2}}.
\]
Taking $R>1$ large enough, we see that 
\[
C\|\chi(D_{\xi})\hat{u}\|^2_{H^{-\theta/2}}\le \frac{\gamma^2}{2}\|\chi(D_{\xi})\hat{u}\|^2_{L^2}.
\]
Here $C$ and thus $R$ are independent of $M>1$ by the uniformity mentioned above. 
This implies that 
\[
\|e^{\rho_M(x)}(H-z)e^{-\rho_M(x)}\chi(x)u\|_{\ell^2}\ge \frac{\gamma}{2}\|\chi(x)u\|_{\ell^2}
\]
for any $u\in \ell^2(\ZZ^d)$. This implies that 
\[
\|e^{\rho_M(x)}(H-z)\chi(x)u\|_{\ell^2}
\ge \frac{\gamma}{2}\|e^{\rho_M(x)}\chi(x)u\|_{\ell^2}
\]
for any $u\in \ell^2(\ZZ^d)$. 
Taking the limit $M\to \infty$, we have 
\[
\|e^{(1-\ve)\rho_E(x)}(H-z)\chi(x)u\|_{\ell^2}
\ge \frac{\gamma}{2}\|e^{(1-\ve)\rho_E(x)}\chi(x)u\|_{\ell^2}
\]
for any $u \in \ell^2(\ZZ)$. 
Then Theorem~\ref{thm-3} follows if 
we calculate the commutator as in the proof of Theorem~\ref{thm-2} and take
$\tilde{\chi}\in \ell^{\infty}_{\mathrm{comp}}$ which is $1$ on $\{x\in \ZZ^d| \, R-1<|x|<R+2\}$.  
\end{proof}

\subsection{The optimality of Theorem~\ref{thm-3}}
We prove that the exponential decay of eigenfunctions in Corollary~\ref{coro1} 
is optimal for a simple discrete Schr\"{o}dinger operator. 
Fix any $E<0$ and define $u_E\in \ell^2(\ZZ^d)$ by 
\[
u_E(x)=(2\pi)^{-d}\int_{\TT^d}
\bigl(4\sum_{j=1}^d \sin^2 \frac{\xi_j}{2}+|E|\bigr)^{-1}e^{-i\langle x, \xi\rangle} d\xi.
\]
Then $(H_0+|E|)u_E(x)=\delta_0(x)$, 
where $H_0$ is the free discrete Schr\"{o}dinger operator and $\delta_0$ is the 
delta function supported on $0\in \ZZ^d$. 
We note that $u_E(0)>0$. 
Thus if we set $V(x)=-u_E(0)^{-1}\delta_0(x)$, we have $(H_0+V)u_E(x)=Eu_E(x)$. 
We study the exponential decay of this eigenfunction $u_E$. 
We note that Corollary~\ref{coro1} for $u_E$ follows from the deformation of the integral in 
the definition of $u_E$. 

Take a bounded domain $0\in \Omega\subset \RR^d$ 
and set 
\[
\rho_{\Omega}(x)=\sup_{\xi \in \Omega}\langle x, \xi \rangle. 
\] 
The following proposition gives the optimality of Theorem~\ref{thm-3} and Corollary~\ref{coro1}. 
Recall that 
$K^E=\{\xi\in \RR^d |\, 4\sum_{j=1}^d \sinh^2\frac{\xi_j}{2}\le|E|\}$.

\begin{prop}
Under the above notation, assume that 
\[
|u_E(x)|\le Ce^{-\rho_{\Omega}(x)}
\]
for some $C>0$ and any $x \in \ZZ^d$. Then $\Omega\subset K^{E}$.
\end{prop}
\begin{proof}
The Fourier inversion formula implies 
\[
(4\sum_{j=1}^d \sin^2 \frac{\xi_j}{2}+|E|\bigr)^{-1}=\sum_{x\in \ZZ^d}u_E(x)e^{i\langle x, \xi \rangle}.
\]
The assumption on $u_E$ implies that 
\[
|u_E(x)e^{i\langle x, \xi \rangle}|\le C e^{-\rho_{\Omega}(x)}
e^{-\langle\mathrm{Im}\,\xi, x\rangle}. 
\]
This implies that $(4\sum_{j=1}^d \sin^2 \frac{\xi_j}{2}+|E|\bigr)^{-1}$ has an analytic continuation to 
$\{\xi\in \CC^d/2\pi \ZZ^d|\,-\mathrm{Im}\,\xi \in \Omega\}$. 
Since $4\sin^2 \xi_j/2=-4\sinh^2\mathrm{Im}\,\xi_j/2$ for $\mathrm{Re}\,\xi_j=0$, 
this implies $\Omega\subset K^E$.
\end{proof}

\begin{remark}
For $d=1$, it is known that $u_E(x)=(|E|(4+|E|))^{-1/2} e^{-\rho_E(x)}$ (\cite[Theorem~2.2]{ItJ}). 
Ito-Jensen~\cite[Theorem~2.1, 2.4]{ItJ} proved that $u_E$ is expressed 
by a hypergeometric function of several variables for $d\ge 2$ 
and by a generalized hypergeometric function of one variable for $d=2$. 
The precise asymptotics of $u_E(x)$ when $|x|\to \infty$ does not seem to be immediate from these expressions.
\end{remark}

\section*{Acknowledgement}
The author is grateful to Shu Nakamura and Kenichi Ito for helpful discussions and encouragement. 
The author 
is supported by JSPS KAKENHI Grant Number JP21J10860.

Graduate School of Mathematical Sciences, the University of Tokyo,
 3-8-1, \\Komaba, Meguro-ku, Tokyo 153-8914, Japan

E-mail address: kameoka@ms.u-tokyo.ac.jp

\end{document}